\documentclass[a4paper,reqno,12pt]{amsart}
\usepackage[a4paper,margin=1in]{geometry}
\usepackage{graphicx}
\usepackage{amsmath,amssymb,amsthm}
\usepackage{hyperref}

\theoremstyle{plain}
\newtheorem{theorem}{Theorem}[section]
\newtheorem{lemma}[theorem]{Lemma}
\newtheorem{proposition}[theorem]{Proposition}

\newtheorem{corollary}[theorem]{Corollary}
\theoremstyle{definition}
\newtheorem{remark}[theorem]{Remark}
\numberwithin{equation}{section}

\title[The Frechet derivatives of stratified steady Stokes waves]
{Spectral properties of the Frechet derivatives of stratified steady Stokes waves}

\author{Vladimir Kozlov$^a$$^b$}

\address{$^a$Department of Mathematics, Link\"oping University, Link\"oping, Sweden}
\address{$^b$ Euler International Mathematical Institute (EIMI), Saint Petersburg, Russia}

\begin{document}
	
\begin{abstract}
We consider stratified steady water waves in a two dimensional channel. Our main
subject is spectral properties of the Frechet derivatives of steady
water Stokes waves. One of main results is the absence of subharmonic water waves
 in a neighborhood of a Stokes wave. The main
assumption is formulated in terms
of the eigenvalues of the Frechet derivative evaluated at this wave and considered
in the class of periodic solutions of the same period. The first
eigenvalue is always negative. We show that if
 the second eigenvalue is positive then there are no waves with multiple periods
in a neighborhood of the Stokes wave.


\end{abstract}

\maketitle

\section{Introduction}

We consider stratified steady water waves in a two-dimensional channel.
We use the classical formulation of the problem based on the Euler equations.
The surface tension is neglected and the water motion can be rotational. The aim of
this paper is to present some results on uniqueness of water waves.

The problem of uniqueness for water waves is connected with the Benjamin-Lighthill
conjecture, which suggests
estimates for the main parameters of the problem, such that the volume flux, the Bernoulli
constant  and the flow force.  In the case of incompressible and irrotational fluid such
estimates were proved in \cite{BL54}, \cite{B95} and \cite{KN75}. For incompressible and
rotational fluid such estimates were obtained \cite{Lok1}, where the main references in
this field can be found. We can mention here also the papers \cite{KN13}, \cite{KN13a} and
\cite{KLN17}, where uniqueness results are proved for near critical Bernoulli constant
for  incompressible fluid.

Here we study the following local uniqueness property. Assume that we have a certain
periodic, even solution to the water waves problem. Is this solution unique in a small
neighborhood?
We give certain conditions on the first eigenvalues of the Frechet derivative, which
guarantee such uniqueness for periodic waves in the class of solutions with
multiple periods.

\subsection{Statement of the main result}\label{SOkt11a}

Our object of study is two-dimensional stratified steady water waves  traveling
with constant speed $c$ under the influence of gravity. To eliminate the dependence on
time we use a moving reference frame, where the fluid occupies a  domain
 $$
D=D_\xi=\{(x,y)\,:\,-d<y<\xi(x),\;x\in\Bbb R\}
$$
in the channel with  the flat bottom $B$ given by $y=-d$ and with the free surface of
the flow $S=S_\xi$ given by $y=\xi(x)$. The density of the fluid $\rho$, defined
 in $\overline{D}$, is assumed to be positive and not necessarily constant.
 To describe a water wave we use the velocity of the flow $(u,v)$, the pressure $P$ 
 and the density $\rho$.
 Corresponding relations can be found in the paper Walsh \cite{Wal}. We recall these
 relations for readers convenience:
\begin{equation}\label{Se17a}
u_x+v_y=0\;\;\mbox{in $D$}\;\;\mbox{(incompressibility)}
\end{equation}
\begin{equation}\label{M3a}
(u-c)\rho_x+v\rho_y=0\;\;\mbox{in $D$}\;\;\mbox{(the conservation of mass)}
\end{equation}
and the conservation of momentum
\begin{eqnarray}\label{Se17aa}
&&(u-c)u_x+vu_y+\frac{P_x}{\rho}=0\;\;\mbox{in $D$},\nonumber\\
&&(u-c)v_x+vv_y+\frac{P_y}{\rho}=-g\;\;\mbox{in $D$},
\end{eqnarray}
where
 $g$ is the gravitational constant. 
The boundary conditions are
\begin{equation}\label{Sep18a}
v=(u-c)\xi_x\;\;\mbox{and}\;\;P=P_{atm}\;\;\mbox{on $S$},
\end{equation}
where $P_{atm}$ is the atmospheric pressure, and
\begin{equation}\label{Sep18aa}
v=0\;\;\mbox{on $B$}.
\end{equation}

The pseudostream function $\psi=\psi(x,y)$ is defined by
$$
\psi_x(x,y)=-\sqrt{\rho}v(x,y),\,\;\psi_y(x,y)=\sqrt{\rho}(u(x,y)-c).
$$
Then  equation  (\ref{Se17a}) is satisfied. The  relative pseudomass is
$$
p_0=\int_{-d}^{\xi(x)}\sqrt{\rho}(u(x,y)-c)dY
$$
and it does not depend on $x$.
From (\ref{M3a}) it follows that the function $\rho$ is constant along
the stream lines. If
\begin{equation}\label{Ma13a}
\psi_y\leq 0\;\;\mbox{in $\overline{D}$},
\end{equation}
then $\rho$ can be written as
\begin{equation}\label{Ma13aa}
\rho=\rho(-\psi).
\end{equation}
In what follows we will assume that $p_0<0$ and that the function $\rho$ is a positive
function of $-\psi$  even in the case, when $\psi$ is not monotone with respect to $y$.

One can verify  that the quantity
\begin{equation}\label{J25a}
E:=\frac{1}{2}(\psi_x^2+\psi_y^2)+ P+g\rho y
\end{equation}
is constant along the stream lines of $\psi$. This allows  to define
 the Bernoulli function
\begin{equation}\label{Ma13b}
\beta(\psi)=\frac{d E}{d\psi}.
\end{equation}
It can be verified that
\begin{equation}\label{Ma13ba}
\beta (\psi)=\Delta\psi-gy\rho'(-\psi)=0\;\;\mbox{in $D$}.
\end{equation}
Boundary conditions for $\psi$ are
\begin{equation}\label{Ma13bb}
\psi(x,\xi(x))=0\;\;\mbox{and}\;\;\psi(x,-d)=-p_0
\end{equation}
together with the Bernoulli boundary condition
\begin{equation}\label{Ma13bc}
\frac{1}{2}|\nabla\psi|^2+g\rho(0)(\xi(x)+d)=R,
\end{equation}
which is obtained from (\ref{J25a}) by setting $y=\xi(x)$ and using that
$\psi(x,\xi(x))=0$. The constant $R$ is called the Bernoulli constant.

We set
\begin{equation}\label{Okt5a}
\omega(y,\psi)=-gy\rho'(-\psi)-\beta(\psi).
\end{equation}
Then equation (\ref{Ma13ba}) takes the form
\begin{equation}\label{Okt5aa}
\Delta\psi+\omega(y,\psi)=0\;\;\mbox{in $D$}.
\end{equation}

In order to formulate smoothness assumptions, let us introduce some functional spaces.
For $k=0,1,\ldots$, $\alpha\in(0,1)$ and $-\infty<a<b<\infty$
we denote by $C^{k,\alpha}(\overline{D}_{a,b})$ and $C^{k,\alpha}([a,b])$
the H\"older spaces of functions in $\overline{D}_{a,b}=\{(x,y)\in \overline{D}
\,:\,a\leq x\leq b\}$ and $[a,b]$ respectively.

The space $C^{k,\alpha}_{\textrm{b}}(\overline{D})$ consists of functions
$u$ defined on $\overline{D}$
such that
$$
\sup_{a\in \Bbb R}||w||_{C^{k,\alpha}(\overline{D}_{a,a+1})}<\infty.
$$
 The subspaces   $C^{k,\alpha}_{\textrm{per}}(\overline{D})$,
$C^{k,\alpha}_{0,e,\textrm{per}}(\overline{D})$  consists of
$\Lambda$-periodic ($\Lambda$--periodic, even) functions in
$C^{k,\alpha}_{\textrm{b}}(\overline{D})$ vanishing for $y=-d$.

Similarly, one can define the spaces $C^{k,\alpha}_{\textrm{b}}(\Bbb R)$,
$C^{k,\alpha}_{\textrm{per}}(\Bbb R)$ and
$C^{k,\alpha}_{e,\textrm{per}}(\Bbb R)$.

We will assume that the density function $\rho$ is of class $C^{2,\alpha}$
for a certain $\alpha\in (0,1)$ and the Bernoulli function $\beta$ is of class
$C^{1,\alpha}$. In what follows the solution
$\psi\in C^{2,\alpha}_{0,e,\textrm{per}}(\overline{D})$
is assumed to be fixed. We always assume that $\psi_y\neq 0$ on the free surface $S_\xi$,
which together with (\ref{Ma13bb}) and (\ref{Ma13bc})  implies
 $\xi\in C^{2,\alpha}_{e,\textrm{per}}(\Bbb R)$.

\begin{remark} The functions $\rho$ and $\beta$ are more or less arbitrary and are chosen
to model some important properties of the flow under consideration. For example to model
a unidirectional flow we assume that
\begin{equation}\label{Ma18a}
u-c<0\;\;\mbox{in $\overline{D}$},
\end{equation}
which implies
$$
\psi_y<0\;\;\mbox{inside $\overline{D}$}.
$$
In this case $p_0$ is negative and the function $\rho$ can be chosen to guarantee
the monotonicity of the function $\rho(-\psi)$.

The next example involves a linear function $\rho$ with respect to $\psi$, and a constant
$\beta$, see \cite{Esch}. This choice allows
to include in consideration flows with counter-currents.
\end{remark}

\subsection{The Frechet derivative}\label{SOkt13ab}

Let us present a formal derivation of the Frechet derivative for the nonlinear operator in
the problem (\ref{Ma13ba})--(\ref{Ma13bc}) at a solution $(\hat{\psi},\hat{\xi})$.
We consider a small perturbation of this solution
$$
\psi=\hat{\psi}+tu,\;\;\xi=\hat{\xi}+t\tilde{\xi}.
$$
Inserting these functions in (\ref{Ma13ba})--(\ref{Ma13bc}) and collecting terms
of order $t$, we get the following expressions
\begin{equation}\label{Se28a}
\Delta u+gy\rho^{''}(-\hat{\psi})u-\beta' (\hat{\psi})u\;\;\mbox{in $D$},
\end{equation}
\begin{equation}\label{Se28ab}
\nabla\hat{\psi}\cdot\nabla u+\partial_y|\nabla\hat{\psi}|^2\tilde{\xi}
+g\rho(0)\tilde{\xi}(x),
\end{equation}
defined on the functions subject to
\begin{equation}\label{Se28aa}
u(x,\hat{\xi}(x))+\hat{\psi}_y(x,\hat{\xi}(x)\tilde{\xi}=0\;\;\mbox{and}\;\;u(x,-d)=0.
\end{equation}
From the first relation in (\ref{Se28aa}) we find
$$
\tilde{\xi}=-\frac{u}{\hat{\psi}_y}\;\;\mbox{for $y=\hat{\xi}(x)$}
$$
and rewrite the expressions  (\ref{Se28a}) and (\ref{Se28ab}) as
\begin{eqnarray}\label{Se28b}
&&A(x,y,\partial_x,\partial_y)u=A u:=\Delta u+gy\rho^{''}(-\psi)u-\beta' (\psi)u\;\;
\mbox{in $D$},\nonumber\\
&&B(x,y,\partial_x,\partial_y)u=B u:=\nabla\hat{\psi}\cdot\nabla u-\sigma(x)u\;\;
\mbox{for $y=\hat{\xi}(x)$}
\end{eqnarray}
defined on functions subject to
\begin{equation}\label{Se28ad}
u(x,-d)=0.
\end{equation}
Here
\begin{equation}\label{Se28ba}
\sigma(x)=\frac{\nabla\hat{\psi}\cdot\partial_y\nabla\hat{\psi}+g\rho(0)}
{\hat{\psi}_y},\;\;\mbox{where $y=\hat{\xi}(x)$}.
\end{equation}
In what follows we will use the notation
\begin{equation}\label{Okt6b}
\omega_*=\omega_*(x,y)=gy\rho^{''}(-\hat{\psi}(x,y))-\beta' (\hat{\psi}(x,y)).
\end{equation}
Our smoothness assumptions lead to
$$
\omega_*\in C^{0,\alpha}_{0,e,\textrm{per}}(\overline{\Omega})\;\;\mbox{and}\;\;
\sigma\in C^{0,\alpha}_{0,e,\textrm{per}}(\Bbb R).
$$

There are several spectral problems connected with the Frechet derivative. The first one,
 which is the main subject of the paper, is the following
 \begin{eqnarray}\label{Se28bd}
&&((\partial_x+z)^2+\partial_y^2)
u+\omega_*(x,y)u=0\;\;\mbox{in $D$},\nonumber\\
&&\nabla\hat{\psi}\cdot (\partial_x+z,\partial_y) u-\sigma(x)u=0\;\;
\mbox{for $y=\hat{\xi}(x)$},
\end{eqnarray}
 where $z\in\Bbb C$ is a spectral parameter.
  Here the function $u$ is $\Lambda$-periodic and subject to (\ref{Se28ad}).
  We do not assume that the function $u$ is even in this spectral problem.
  If $(z,u)$ solves the problem then the pair $(z+ik\tau_*,e^{-k\tau_*x}u)$, where
  $\tau_*=2\pi/\Lambda$ and $k$ is an integer, is also
  a solution to this problem. The spectral parameter $z$ is usually called quasimomentum.
This problem is important in study of asymptotic behaviour of solutions of boundary
value problem with
periodic coefficients (see \cite{N1}) and corresponding nonlinear problems
(see \cite{KT})). The problem (\ref{Se28bd}) appears when we are looking for the Floquet
solutions
$$
U(x,y)=e^{zx}\sum_{k=0}^Na_k(x,y)x^k
$$
of the problem
 \begin{eqnarray}\label{Okt26a}
&&(\partial_x^2+\partial_y^2)
U+\omega_*(x,y)U=0\;\;\mbox{in $D$},\nonumber\\
&&\nabla\hat{\psi}\cdot\nabla U-\sigma(x)U=0\;\;
\mbox{for $y=\hat{\xi}(x)$},\nonumber\\
&&U(x,-d)=0.
\end{eqnarray}
Here $a_k$ are $\Lambda$--periodic coefficients.

 The second spectral problem is
 \begin{eqnarray}\label{Se28be}
&&\Delta u+\omega_*(x,y)u+\mu u=0\;\;\mbox{in $D$},\nonumber\\
&&\nabla\hat{\psi}\cdot\nabla u-\sigma(x)u=0\;\;\mbox{for $y=\hat{\xi}(x)$},\nonumber\\
&&u(x,-d)=0,
\end{eqnarray}
This problem is equivalent to the spectral problem for the Frechet derivative after
application of partial hodograph transformation, which can be used for
unidirectional flows, i.e. $\psi_y\neq 0$ inside $\overline{D}$, see Sect. \ref{SOkt13a}
and Sect. \ref{SOkt13aa}.

Finally, the third spectral problem has the form
\begin{eqnarray}\label{Se28bf}
&&\Delta u+\omega_*(x,y)u=0\;\;\mbox{in $D$},\nonumber\\
&&\nabla\hat{\psi}\cdot\nabla u-\sigma(x)u=\theta u\;\;\mbox{for $y=\hat{\xi}(x)$},
\nonumber\\
&&u(x,-d)=0.
\end{eqnarray}
This problem appears in the case $\psi_y\neq 0$ on $S$, see Sect. \ref{SOkt13b}.

Usually the Frechet derivatives are defined for operators acting in fixed spaces.
In order to define the Frechet derivatives for operators connected to a
free boundary value problem it is useful to reduce it to a problem defined
in a fixed domain. It can be done by using a flattening change of variables.
We consider two cases. The first one serves for unidirectional flows, i.e.
$\hat{\psi}_y <0$ in $\overline{D}$. In this case we can use the partial hodograph
transformation. The second case covers the problems when
$\hat{\psi}_y(x,\hat{\xi}(x))<0$.

\subsection{Results}

In Sect. \ref{SOkt10b} we prove that the spectral problems (\ref{Se28be}) and (\ref{Se28bf})
has the same number of negative eigenvalues. Namely, let
\begin{equation}\label{Okt5b}
{\bf a}(u,v)=
\int_\Omega (\nabla u\cdot\nabla\overline{v}-\omega_*u\overline{v})dxdy
-\int_{-\Lambda/2}^{\Lambda/2}\sigma(x)u(x,\xi(x))\overline{v}(x,\xi(x))\frac{dx}{\psi_y},
\end{equation}
where
\begin{equation}\label{Okt6a}
\Omega=\{(x,y)\,:\,x\in (-\Lambda/2,\Lambda/2),\;y\in (-d,\xi(x))\}.
\end{equation}
We assume that
\begin{equation}\label{Okt5ba}
a(u,u)>0\;\;\mbox{for all nonzero $u\in \hat{H}^1(\Omega)$},
\end{equation}
where $\hat{H}^1(\Omega)$ consists of functions $u$ in $H^1(\Omega)$ such that
$u(-\Lambda/2,y)=u(\Lambda/2,y)$ and $u(x,-d)=u(x,\xi(x))=0$.

Inequality (\ref{Okt5ba}) is true in the case of a unidirectional flow or a flow
with constant vorticity and linear density.

\begin{proposition}\label{Pm23} Assume that (\ref{Okt5ba}) is valid.
Let $a$ and $b$ be continuous, even, positive
and  $\Lambda$-periodic functions defined in $\overline{D} $ and
in $\Bbb R$ respectively.
Consider the spectral problems
\begin{eqnarray}\label{F20aa}
&&Au+\mu au=0\;\;\mbox{in $\Omega$,}\nonumber\\
&&B u=0\;\;\mbox{for $y=\xi(x)$,}\nonumber\\
&&u=0\;\;\mbox{for $p=-d$}
\end{eqnarray}
and
\begin{eqnarray}\label{F20a}
&&A u=0\;\;\mbox{in $\Omega$,}\nonumber\\
&&B u=\theta bu\;\;\mbox{for $y=\xi(x)$,}\nonumber\\
&&u=0\;\;\mbox{for $y=-d$},
\end{eqnarray}
where $\mu$ and $\theta$ are spectral parameters and $u$ is an even, $\Lambda$--periodic
function. Then these problems
 have the same number of negative eigenvalues
(accounting their multiplicities). Moreover, if we consider these spectral problems
for $M\Lambda$-periodic, even functions, then they have the same number
of negative eigenvalues also.
\end{proposition}

The problem (\ref{Se28be}), considered for even, $\Lambda$--periodic functions,
 is a spectral problem for self-adjoint
 operator and its spectrum
consists of isolated eigenvalues of finite multiplicity, which are uniformly bounded
from below. We numerate the eigenvalues in non-decreasing order
\begin{equation}\label{Okt11a}
\mu_1\leq\mu_2\leq\dots,
\end{equation}
taking into account the multiplicities of the eigenvalues.

 It was proved in Lemma
\ref{AL1} that the first eigenvalue $\mu_1$ is always negative.
The main our assumption is
\begin{equation}\label{Mabb}
\mbox{the second eigenvalue $\mu_2$ is positive}.
\end{equation}
By Proposition \ref{Pm23} both spectral problems (\ref{Se28be}) and (\ref{Se28bf})
has one negative eigenvalue and the second eigenvalue is positive. Therefore we choose
the spectral problem (\ref{Se28be}) as the basic auxiliary problem for investigating spectrum of
(\ref{Se28bd}). We put
\begin{equation}\label{Okt23a}
\tau_*=\frac{2\pi}{\Lambda}.
\end{equation}
Our main concern is the spectral problem (\ref{Se28bd}).
The main result about the spectral problem (\ref{Se28bd}) is the following

\begin{theorem}\label{TOkt13a} Let (\ref{Mabb}) be valid and $\hat{\psi}_x\neq 0$ in
\begin{equation}\label{Okt27a}
\Omega_+=\{(x,y)\,:\,x\in (0,\Lambda/2),\;y\in (-d,\hat{\xi}(x))\}.
\end{equation}
Then the set $z=i\tau$, $\tau\in [0,\tau_*)$, contains the
only eigenvalue $z=0$ of the problem (\ref{Se28bd}) which is simple with the odd
eigenfunction $\hat{\psi}_x$.
\end{theorem}

The proof is given in Sect. \ref{SOkt22a}. The next important corollary is proved
in Sect. \ref{CorOkt22a}.

\begin{corollary}\label{CorOkt22} Consider the spectral problem in the space of 
$m\Lambda$--periodic functions, where
$m$ is an odd integer. Then there are no $m\Lambda$ periodic, even solutions in a small
neighborhood of $(\hat{\psi},\hat{\xi})$.

\end{corollary}

\section{Proofs}

In Sect. \ref{SOkt13a}, \ref{SOkt13aa} and \ref{SOkt13b} we prove an equivalence
of the formal Frechet derivative, constructed in Sect. \ref{SOkt13ab}, with
the real Frechet derivatives, which are calculated after changes of variables. We
consider two changes of variables: the partial hodograph transformation, which can be
 applied when
$\psi_y<0$ in $\overline{D}$, and a flattening  the domain applying in the case 
when $\psi_y\neq 0$
 on $S$. It appears that corresponding spectral problems has a spectral parameter in
 different places, the first in the equation of the problem and the second in the boundary
 condition. In Sect. \ref{SOkt10b} we show that these problems have the same number
 of negative eigenvalues. The proof of Theorem \ref{TOkt13a}   is based on the estimates
 of eigenvalues of spectral problems for the Frechet derivatives obtained in Sect.
 \ref{SOkt25a}. Proofs of Theorem \ref{TOkt13a} and Corollary \ref{CorOkt22}
 are presented in Sect. \ref{SOkt22a} and  \ref{CorOkt22a}  respectively. Finally, in
 Sect. \ref{SOkt25b} we discuss generalized eigenfunctions corresponding to
 the eigenvalue $z=0$ of the spectral problem (\ref{Se28bd}).

\subsection{Partial hodograph transformation}\label{SOkt13a}

Here we assume that
\begin{equation}\label{Se28c}
\hat{\psi}_y <0\;\;\mbox{ in $\overline{D}$}
\end{equation}
and we will use the notation $\psi$ instead of $\hat{\psi}$
Let
$$
q=x,\;\;p=-\psi(x,y),\;\;(x,y)\in \overline{D}.
$$
Then
$$
(q,p)\in\overline{Q},\;\;Q=\{q\in\Bbb R,\;p\in (p_0,0)\}.
$$
We put
$$
h(q,p)=d+y.
$$
Then
$$
\psi_y=-\frac{1}{h_p},\;\;\psi_x=\frac{h_q}{h_p}.
$$
Applying this change of variables, we arrive at
\begin{eqnarray}\label{J17aa}
&&F(h):=\Big(\frac{1+h_q^2}{2h_p^2}\Big)_p-\Big(\frac{h_q}{h_p}\Big)_q+g(h-d)\rho_p
-\beta(-p)=0\;\;
\mbox{in $Q$},\nonumber\\
&&G(h):=\frac{1+h_q^2}{2h_p^2}+gh\rho-R=0\;\;\mbox{for $p=0$},\nonumber\\
&&h(q,p_0)=0\;\;q\in\Bbb R.
\end{eqnarray}

From our smoothness assumptions, formulated at the end of Sect. \ref{SOkt11a}, it follows
 that $h$ is an even, $\Lambda$--periodic function from the space
$C^{2,\alpha}_{\textrm{b}}(\overline{Q})$ with
\begin{equation}\label{J25aa}
\delta\leq h_p(q,p)\;\;\mbox{on $\overline{Q}$ and $\rho\in C^{1,\alpha}([p_0,0])$,
$\beta\in C^{0,\alpha}([p_0,0])$},
\end{equation}
where $\delta$ is a positive constant.
The nonlinear operators in (\ref{J17aa}) is continuous in the following spaces
$$
(F,G)\,:\,C^{2,\alpha}_{0,e,\textrm{per}}(\overline{Q})\rightarrow
C^{0,\alpha}_{0,e,\textrm{per}}(\overline{Q})\times
C^{1,\alpha}_{e,\textrm{per}}(\Bbb R).
$$
Here we consider only positive functions $h_p$.

Now the spaces for the operators $F$ and $G$ are fixed and the Frechet derivatives
of the operators in (\ref{J17aa}) are defined in a usual way;
\begin{eqnarray}\label{J25ab}
&&{\mathcal F}w:=\Big(\frac{h_qw_q}{h_p^2}-\frac{(1+h_q^2)w_p}{h_p^3}\Big)_p
-\Big(\frac{w_q}{h_p}-\frac{h_qw_p}{h_p^2}\Big)_q+gw\rho_p\;\;\mbox{in $Q$},\nonumber\\
&&{\mathcal G}w:=\frac{h_qw_q}{h_p^2}-\frac{(1+h_q^2)w_p}{h_p^3}+g\rho w\;\;
\mbox{for $p=0$},\nonumber\\
&&w(q,p_0)=0\;\;q\in\Bbb R,
\end{eqnarray}
where $w\in C^{2,\alpha}_{0,e,\textrm{per}}(\overline{Q})$.

\subsection{The Frechet derivative (\ref{J25ab}) in $(x,y)$ variables}\label{SOkt13aa}

We start from a function $u$ depending on $(x,y)\in\overline{D}$ and put
\begin{equation}\label{J17aa}
w(q,p)=u(q,h(q,p))h_p(q,p).
\end{equation}
To evaluate ${\mathcal F} w$ and ${\mathcal G}w$, we note first that
\begin{equation*}
w_q=u_xh_p+u_yh_qh_p+uh_{qp},\;\;w_p=u_yh_p^2+uh_{pp}.
\end{equation*}
Therefore
\begin{eqnarray}\label{Ok1aa}
&&I_1:=\frac{h_qw_q}{h_p^2}-\frac{(1+h_q^2)w_p}{h_p^3}=
\frac{h_q(u_xh_p+u_yh_qh_p+uh_{qp})}{h_p^2}
-\frac{(1+h_q^2)(u_yh_p^2+uh_{pp})}{h_p^3}\nonumber\\
&&=u_x\frac{h_q}{h_p}-u_y\frac{1}{h_p}+\Big(\frac{1+h_q^2}{2h_p^2}\Big)_pu
=\psi_xu_x+\psi_yu_y-\frac{1}{2\psi_y}(\psi_x^2+\psi_y^2)_y u
\end{eqnarray}
and
\begin{eqnarray*}
&&I_2:=\frac{(u_xh_p+u_yh_qh_p+uh_{qp})}{h_p}-\frac{h_q(u_yh_p^2+uh_{pp})}{h_p^2}=
u_x+u\Big(\frac{h_q}{h_p}\Big)_p\\
&&=u_x+h_p\psi_{xy} u.
\end{eqnarray*}
This implies
\begin{eqnarray*}
&&I_{1p}-I_{2q}=-\frac{1}{\psi_y}\Big(\psi_xu_x+\psi_yu_y-\frac{u}{\psi_y}(\psi_x\psi_{xy}
+\psi_y\psi_{yy})\Big)_y\\
&&-(\partial_x-\frac{\psi_x}{\psi_y}\partial_y)(u_x-\frac{\psi_{xy}}{\psi_y}u).
\end{eqnarray*}
Taking all terms containing $u$, we get
$$
\frac{u}{\psi_y}\Big(\frac{\psi_x\psi_{xy}
+\psi_y\psi_{yy}}{\psi_y}\Big)_y
+u(\partial_x-\frac{\psi_x}{\psi_y}\partial_y)\Big(\frac{\psi_{xy}}{\psi_y}\Big)
=\frac{u}{\psi_y}(\psi_{xx}+\psi_{yy})_y.
$$
Taking terms containing the first derivatives of $u$, we obtain
\begin{eqnarray*}
&&-\frac{1}{\psi_y}\Big(\psi_{xy}u_x+\psi_{yy}u_y-\frac{u_y}{\psi_y}(\psi_x\psi_{xy}
+\psi_y\psi_{yy})\Big)\\
&&+\frac{\psi_{xy}}{\psi_y}(\partial_x-\frac{\psi_x}{\psi_y}\partial_y)u=0.
\end{eqnarray*}
Finally, collecting all terms containing second derivatives of $u$, we get
$$
-\frac{1}{\psi_y}\Big(\psi_xu_{xy}+\psi_yu_{yy}\Big)
-(\partial_x-\frac{\psi_x}{\psi_y}\partial_y)u_x=-u_{xx}-u_{yy}.
$$
Using these calculations, we obtain
\begin{equation}\label{Ok1ab}
I_{1p}-I_{2q}=\frac{u}{\psi_y}(\psi_{xx}+\psi_{yy})_y-u_{xx}-u_{yy}.
\end{equation}

Now we can give a connection between the Frechet derivatives (\ref{Se28b})
and (\ref{J25ab}).
\begin{lemma}\label{CorM1} Let $\psi$ and $\xi$ satisfy  equations (\ref{Ma13ba})
and (\ref{Ma13bc}). Let also $F$ and $g$ continuous, even and
$\Lambda$-periodic functions. If $u=u(x,y)$ satisfies
 the problem
 \begin{eqnarray}\label{Ok1a}
 &&A u=f\nonumber\\
 &&B u=g
 \end{eqnarray}
 where the operators $A$ and $B$ are defined by (\ref{Se28b}),
then the function $w$ given by (\ref{J17aa}) satisfies the equations
\begin{eqnarray}\label{J18ba}
&&{\mathcal F}w=-f(q,h),\nonumber\\
&&{\mathcal G}w=g\;\;\mbox{for $p=0$},\nonumber\\
&&w=0\;\;\mbox{for $p=p_0$}.
\end{eqnarray}
\end{lemma}
\begin{proof} Using (\ref{Ok1ab}), we get
\begin{equation}\label{Ok1b}
{\mathcal F}w=\frac{u}{\psi_y}(\psi_{xx}+\psi_{yy})_y-u_{xx}-u_{yy}-gw\rho_p
\end{equation}
Differentiating (\ref{Okt5aa}) with respect to $y$, we obtain
$$
\Delta\psi_y-g\rho'(-\psi)+g\rho^{''}(-\psi)\psi_y-\beta' (\psi)\psi_y=0,
$$
which allows to rewrite (\ref{Ok1b}) as
$$
{\mathcal F}w=\frac{u}{\psi_y}(g\rho'(-\psi)-gy\rho^{''}(-\psi)\psi_y
+\beta' (\psi)\psi_y)-u_{xx}-u_{yy}+gw\rho_p.
$$
Thus
$$
{\mathcal F}w=u(-gy\rho^{''}(-\psi)+\beta' (\psi))-u_{xx}-u_{yy}.
$$
This proves the first relation in (\ref{J18ba}).

Furthermore,
$$
{\mathcal G}w=I_1+g\rho w=\nabla\psi\cdot\nabla u-\sigma u=g.
$$
This proves the second relation in (\ref{J18ba}).
\end{proof}

\begin{remark}
This lemma implies: if
$$
f=-\mu u(q,h)\;\;\mbox{and}\;\; g=0\;\;\mbox{then}\;\;
$$
in (\ref{Ok1a}), then
\begin{eqnarray*}
&&{\mathcal F}w=\mu \frac{w}{h_p},\nonumber\\
&&{\mathcal G}w=0\;\;\mbox{for $p=0$}.\nonumber\\
\end{eqnarray*}
Thus we obtain a connection between spectral problems for the Frechet derivatives
 in $(x,y)$ and $(q,p)$ variables.
\end{remark}

\subsection{Flattening the boundary in the second case}\label{SOkt13b}

In the case $\psi_y\neq 0$ only on $S$ we cannot apply the partial hodograph
change of variables.
In this case we  use the following flattening change of variables
$$
\hat{x}=x,\;\;\;\hat{y}=\frac{d(y+d)}{\xi(x)+d},
$$
to reduce the problem to a fix strip
$$
Q=\{(\hat{x},\hat{y})\,:\,\hat{x}\in \Bbb R,\;\;0<\hat{y}<d\}.
$$
Since
$$
\partial_x=\partial_{\hat{x}}-\frac{\hat{y}\xi'}{\xi+d}\partial_{\hat{y}},
\;\;\partial_y=\frac{d}{\xi+d}\partial_{\hat{y}},
$$
where $'$ means $\partial_{\hat{x}}$,
the equations (\ref{Okt5aa}) and  (\ref{Ma13bc}) takes the form
\begin{eqnarray}\label{K2aa}
&&F(\hat{\psi},\xi):=\Big(\Big(\partial_{\hat{x}}-\frac{\hat{y}\xi'}{\xi+d}
\partial_{\hat{y}}\Big)^2+\Big(\frac{d}{\xi+d}\partial_{\hat{y}}\Big)^2\Big)
\hat{\psi}+\hat{\omega}(\hat{y},\hat{\psi})=0\;\;\mbox{in $Q$},\nonumber\\
&&G(\hat{\psi},\xi):=\frac{1}{2}\Big(\Big|\Big(\partial_{\hat{x}}-
\frac{\hat{y}\xi'}{\xi+d}\partial_{\hat{y}}\Big)\hat{\psi}\Big|^2
+\Big|\frac{d}{\xi+d}\partial_{\hat{y}}\hat{\psi}\Big|^2\Big)+g\rho(0)(\xi(x)+d)-R=0
\;\;\mbox{for $\hat{y}=d$},\nonumber\\
&&\hat{\psi}=0\;\;\mbox{for $\hat{y}=d$},\nonumber\\
&&\hat{\psi}=-p_0\;\;\mbox{for $\hat{y}=0$},
\end{eqnarray}
where
$$
\hat{\psi}({\hat{x}},\hat{y})=\psi\Big({\hat{x}},
\frac{\hat{y}(\xi({\hat{x}})+d)}{d}-d\Big)
\;\;\mbox{or}\;\;\psi(x,y)=\hat{\psi}({\hat{x}},\hat{y})
$$
and
$$
\widehat{\omega}(\widehat{y},\widehat{\psi})=\omega(y,\psi(x,y))=\omega\Big(
\frac{\hat{y}(\xi({\hat{x}})+d)}{d}-d,\hat{\psi}(\hat{x},\hat{y})\Big).
$$
Then the problem (\ref{K2aa}) is equivalent to
$$
(F(\hat{\psi},\xi),G(\hat{\psi},\xi))=0,
$$
which is defined on $\Lambda$--periodic, even functions from
$C^{2,\alpha}(Q)\times C^{2,\alpha}(\Bbb R)$  satisfying $\hat{\psi}(\hat{x},0)=-p_0$,
$\hat{\psi}(\hat{x},d)=0$ and $\xi+d>0$.

We calculate the Frechet derivative at $(\hat{\psi},\xi)$:
\begin{eqnarray}\label{Ju28a}
&&{\mathcal F}(u,\zeta):=\partial_tF(\hat{\psi}+tu,\xi+t\zeta)|_{t=0}
=\Big(\Big(\partial_{\hat{x}}-\frac{\hat{y}\xi'}{\xi+d}
\partial_{\hat{y}}\Big)^2+\Big(\frac{d}{\xi+d}\partial_{\hat{y}}\Big)^2\Big)
u+\tilde{\omega}\nonumber\\
&&-\Big(\frac{\zeta}{\xi+d}\Big)'\hat{y}\partial_{\hat{y}}\Big(\partial_{\hat{x}}
-\frac{\xi'}{\xi+d}{\hat{y}}
\partial_{\hat{y}}\Big)\hat{\psi}
-\Big(\partial_{\hat{x}}-\frac{\hat{y}\xi'}{\xi+d}\partial_{\hat{y}}\Big)
\Big(\frac{\zeta}{\xi+d}\Big)'\hat{y}\partial_{\hat{y}}\hat{\psi}\nonumber\\
&&-2\frac{d^2\zeta}{(\xi+d)^3}\partial_{\hat{y}}^2\hat{\psi},
\end{eqnarray}
where
$$
\tilde{\omega}=\omega_\psi(y,\psi)u+(\omega_y+\omega_\psi\psi_y)\hat{y}\frac{\zeta}{d},
$$
and
\begin{eqnarray}\label{Ju28aa}
&&{\mathcal G}(u,\zeta):=\partial_tG(\hat{\psi}+tu,\xi+t\zeta)|_{t=0}=
\Big(\partial_{\hat{x}}\hat{\psi}-\frac{\hat{y}\xi'}{\xi+d}\partial_{\hat{y}}
\hat{\psi}\Big)\Big(\partial_{\hat{x}}u
-\frac{\hat{y}\xi'}{\xi+d}\partial_{\hat{y}}u\Big)\nonumber\\
&&+\frac{d^2}{(\xi+d)^2}\partial_{\hat{y}}\hat{\psi}\partial_{\hat{y}}u+g\rho(0)\zeta
-\Big(\partial_{\hat{x}}\hat{\psi}-\frac{\hat{y}\xi'}{\xi+d}\partial_{\hat{y}}
\hat{\psi}\Big)\Big(\frac{\zeta}{\xi+d}\Big)'\hat{y}\partial_{\hat{y}}\hat{\psi}
\nonumber\\
&&\!\!-\!\frac{d^2\zeta}{(\xi+d)^3}\partial_{\hat{y}}\hat{\psi}=\psi_xu_x+\psi_yu_y
+g\rho(0)\zeta-\psi_x\psi_y\Big(\frac{\zeta}{\xi+d}\Big)_x(\xi+d)\!\!
-\!\psi_y^2\frac{\zeta}{\xi+d}.
\end{eqnarray}
Here $u=0$ for $\hat{y}=0$ and $\hat{y}=d$.


Let us introduce the transformation
\begin{equation}\label{Au2a}
v(x,y)=u(\hat{x},\hat{y})-\psi_{y}(x,y)\frac{(y+d)\zeta}{\xi+d}.
\end{equation}

\begin{lemma} (i) Assume that the functions $\psi$ and $\xi$ satisfy (\ref{Okt5aa})
in the domain $D_\xi$. If the function  $v$ is given
by (\ref{Au2a}) then
$$
(\partial_x^2+\partial_y^2)v+\omega_*v={\mathcal F}(u,\zeta),
$$
where $\omega_*$ is defined by (\ref{Okt6b}).

(ii) Furthermore
$$
\psi_xv_x+\psi_yv_y+\hat{\sigma}\zeta={\mathcal G}(u,\zeta)\;\;
\mbox{on ${\mathcal S}_\xi$},
$$
where
$$
\hat{\sigma}=\psi_x\psi_{xy}+\psi_y\psi_{yy}+g\rho(0).
$$

\end{lemma}
\begin{proof} (i) Using relations (\ref{Ju28a}) and (\ref{Okt5aa}), we get
\begin{eqnarray*}
&&(\partial_x^2+\partial_y^2)v+\omega_*v=\Big(\Big(\partial_{\hat{x}}
-\frac{\hat{y}\xi'}{\xi+d}\partial_{\hat{y}}\Big)^2
+\Big(\frac{d}{\xi+d}\partial_{\hat{y}}\Big)^2 )\Big)
u(\hat{x},\hat{y})+\tilde{\omega}\\
&&-\Big(\partial_x^2+\partial_y^2+\omega_*\Big)\Big(\psi_{y}(x,y)
\frac{(y+d)\zeta}{\xi+d}
\Big)\\
&&=\Big(\Big(\partial_{\hat{x}}-\frac{\hat{y}\xi'}{\xi+d}\partial_{\hat{y}}\Big)^2
+\Big(\frac{d}{\xi+d}\partial_{\hat{y}}\Big)^2 )\Big)u(\hat{x},\hat{y})+\tilde{\omega}
\\
&&-\frac{(y+d)\zeta}{\xi+d}\Big(\partial_x^2+\partial_y^2
+\omega_*\Big)\psi_{y}(x,y)-I,
\end{eqnarray*}
where
\begin{eqnarray*}
&&I=\Big(2(y+d)\psi_{xy}\Big(\frac{\zeta}{\xi+d}\Big)'
+(y+d)\psi_y\Big(\frac{\zeta}{\xi+d}\Big)^{''}\Big)+2\psi_{yy}\frac{\zeta}{\xi+d}.
\end{eqnarray*}
 Comparing this with the second line in (\ref{Ju28a}), we arrive at the assertion (i).

(ii) We have
\begin{eqnarray*}
&&\psi_xv_x+\psi_yv_y+\hat{\sigma}\zeta=\psi_xu_x+\psi_yu_y+\hat{\sigma}\zeta\\
&&-(\psi_x\psi_{xy}+\psi_y\psi_{yy})\zeta-\psi_x\psi_y (y+d)
\Big(\frac{\zeta}{\xi+d}\Big)'-\psi_y^2\frac{\zeta}{\xi+d}\\
&&=\psi_xu_x+\psi_yu_y+g\rho(0)\zeta-\psi_x\psi_y (y+d)\Big(\frac{\zeta}{\xi+d}\Big)'
-\psi_y^2\frac{\zeta}{\xi+d}.
\end{eqnarray*}
This together with (\ref{Ju28aa}) leads to the required proof of (ii).
\end{proof}

\begin{corollary} Let the functions $\psi$ and $\xi$ satisfy the first equation
in (\ref{K2aa}) in the domain $Q$.
Assume that $u$ and $\zeta$ satisfy
\begin{eqnarray}\label{Au2b}
&&{\mathcal F}(u,\zeta)=0\;\;\mbox{in $Q$},\nonumber\\
&&{\mathcal F}(u,\zeta)=\mu b\zeta\;\;\mbox{for $y=d$},\nonumber\\
&&u=0\;\;\mbox{for $y=d$}.
\end{eqnarray}
Then the functions $v$ and $\zeta$ satisfy
\begin{eqnarray}\label{Au2bb}
&& Av=0\;\;\mbox{in $D_\xi$},\nonumber\\
&&B v=\mu b v\;\;\mbox{on $S_\xi$},\nonumber\\
&&v=0\;\;\mbox{for $y=0$}
\end{eqnarray}
and
$$
\zeta=-v/\psi_y \;\;\mbox{on $S_\xi$}.
$$
\end{corollary}

\subsection{Proofs of Proposition \ref{Pm23} and the negativity of $\mu_1$}\label{SOkt10b}
Let us prove Proposition \ref{Pm23}.
\begin{proof} The problems (\ref{F20a}) and
(\ref{F20aa}) present eigenvalue problems for self-adjoint operators and their spectrum
consist of eigenvalues of finite multiplicity with the only accumulation
 point at $\infty$. Introduce spaces
 $$
 {\mathcal X}=\{u\in H^1(\Omega)\,:\,
u(-\Lambda/2,y)=u(\Lambda/2,y)\;\; \mbox{and}\;\; u(x,-d)=0\}
 $$
 and
 $$
 {\mathcal Y}=\{u\in H^1(\Omega)\,:\,
u(-\Lambda/2,y)=u(\Lambda/2,y),\; Au=0\;\; \mbox{and}\;\; u(x,-d)=0\}.
 $$
 Denote by $X$ the finite dimensional subspace in ${\mathcal X}$ generating by
 the eigenfunctions corresponding to
 negative eigenvalues of the problem (\ref{F20aa}). Then $\dim X$ coincides with the number
  of negative eigenvalues counting together with their multiplicities. This number
  also coincides with the dimension of a largest subspace where the form ${\bf a}$
   is negative. The same property is true for the subspace $Y$ consisting
   of eigenfunctions corresponding to the negative eigenvalues of the problem (\ref{F20a}).
   This implies that
   $$
   \dim X\geq \dim Y.
   $$
   To prove the opposite inequality, we proceed as follows. First, we note that
   $$
   {\bf a}(u,v)=0\;\;\mbox{for $u\in {\mathcal Y}$ and $v\in \hat{H}^1(\Omega)$}.
   $$
  For each $w\in {\mathcal X}$ we can find unique $\tilde{w}\in {\mathcal Y}$ such that
  $w=\tilde{w}$ for $y=\xi(x)$. Therefore $\hat{w}:=w-\tilde{w}\in \hat{H}^1(\Omega)$.
Since
$$
{\bf a}(\tilde{w}+\hat{w},\tilde{w}+\hat{w})={\bf a}(\tilde{w},\tilde{w})+
{\bf a}(\hat{w},\hat{w})<0\;\;\mbox{if $w\neq 0$}
$$
and ${\bf a}(\hat{w},\hat{w})\geq 0$ by (\ref{Okt5ba}), we have that
${\bf a}(\tilde{w},\tilde{w})<0$, which prove the opposite inequality.
\end{proof}

Let $\mu_1\leq\mu_2\leq\dots$ be
 the eigenvalues of the spectral problem (\ref{Se28be}) considered for even,
 $\Lambda$--periodic functions, introduced by  (\ref{Okt11a}).

\begin{lemma}\label{AL1} Let $\psi_x\neq 0$ in $\Omega_+$. Then the eigenvalue $\mu_1$
is negative and simple and the corresponding eigenfunction is positive
in $\Omega$.
\end{lemma}
\begin{proof} The simplicity of the lowest eigenvalue and positivity (up to the sign)
of the corresponding eigenfunction $w_1$ are quite standard facts.

Let $v=\psi_x$. By differentiating (\ref{Se28be}) with respect to $x$ one can verify that
the function $v$
satisfies the problem (\ref{Ma13ba}), (\ref{Ma13bc}) with $\mu=0$. The function $v$ is
odd and hence
$v(-\Lambda/2,y)=v(\Lambda/2,y)=v(x,-d)=0$ for $y\in (-d,\xi(\Lambda/2))$. This implies
that $\mu_1\leq 0$. In the case $\mu_1=0$
we get that $v=cw_1$, which implies that $w_1=0$. Hence $\mu_1<0$.

\end{proof}
 \begin{remark}
 Let us discuss the assumption $\psi_x\neq 0$ in $\Omega_+$. Usually large amplitude
 solutions are constructed as elements of a continuous (analytic) branch of solutions
 $(\psi,\xi)$ of (\ref{Ma13ba})--(\ref{Ma13bc}) depending on a parameter $t$ and
 starting from a laminar flow. If the laminar flow satisfies $\psi_x\neq 0$ inside
 $\Omega_+$ then the same is true for all elements on the branch. This fact is well-known,
 see for example \cite{CSst}, \cite{CSrVar} \cite{Koz1Loop} and \cite{W2}.

 \end{remark}

\subsection{Estimates of eigenvalues}\label{SOkt25a}

Here we study eigenvalues of the spectral problem (\ref{Se28be})  in the case when
the function $w$ is not necessary  $\Lambda$-periodic and even. These problems
 will be used in the study our main spectral problem (\ref{Se28bd}).

Let the domain $\Omega$ be defined by (\ref{Okt6a}) and $\Omega_+$ by (\ref{Okt27a}).
Introduce two spectral boundary value problems. The first one is described by
\begin{eqnarray}\label{Ma23aa}
&&\Delta u+\omega_* u+\mu u=0\;\;\mbox{in $\Omega$},\nonumber\\
&&\nabla\hat{\psi}\cdot\nabla u-\sigma(x)u=0\;\;\mbox{for $y=\hat{\xi}(x)$},\nonumber\\
&&u(x,-d)=0,
\end{eqnarray}
together with the Neumann boundary condition on the remaining part of the boundary:
\begin{equation}\label{Ma23ab}
\partial_xw(-\Lambda/2,y)=\partial_xw(\Lambda/2,y)=0\;\;
\mbox{for $y\in (-d,\xi(\Lambda/2))$}.
\end{equation}
For a weak formulation of this problem we introduce some spaces
$$
H^1_0(\Omega)\!=\!\{w\in H^1(\Omega)\,:\,w=0\,\mbox{for}\,y=-d\},\;\;
H^1_{00}(\Omega)\!=\!\{w\in H^1_0(\Omega)\,:\,w=0\,\mbox{for}\,x=\pm\Lambda/2\}.
$$
Then the Neumann spectral problem is formulated as finding of $w\in H^1_0(\Omega)$
and $\mu\in\Bbb R$ satisfying
$$
{\bf a}(w,v)=\mu (w,v)_{L^2(\Omega)}\;\;\mbox{for all $v\in H^1_0(\Omega)$}.
$$
In this formulation we do not assume that the function $w$ is even. Certainly  even,
$\Lambda$--periodic
solutions to the spectral problem (\ref{Se28be}) solve just introduced
spectral problem.
The second spectral problem
is described by the equations (\ref{Ma23aa}) together with the Dirichlet
boundary conditions:
\begin{equation}\label{Ma23ac}
w(-\Lambda/2,y)=w(\Lambda/2,y)=0\;\;\mbox{for $y\in (-d,\xi(\Lambda/2))$}.
\end{equation}
The weak formulation of the Dirichlet spectral problem is the following. Find
$w\in H^1_{00}(\Omega)$ and $\mu\in\Bbb R$ solving the equation
$$
{\bf a}(w,v)=\mu (w,v)_{L^2(\Omega)}\;\;\mbox{for all $v\in H^1_{00}(\Omega)$}.
$$
Here we also do not assume that the function is even or odd, but we note that odd,
$\Lambda$ periodic
solutions to the spectral problem (\ref{Se28be}) deliver solutions to
this problem.
 We denote by $\{\mu_{jN}\}$ and $\{\mu_{jD}\}$, $j=1,\dots$, the eigenvalues for
the Neumann and Dirichlet spectral problems respectively. As usual the numeration
takes into account the multiplicity and the increasing order:
$$
\mu_{1N}\leq \mu_{2N}\leq\cdots,\;\;\mu_{1D}\leq \mu_{2D}\leq\cdots.
$$
Clearly,
$$
\mu_{jN}\leq\mu_j\;\;\mbox{and}\;\;\mu_{jN}\leq\mu_{jD},\;\; j=1,2,\ldots
$$
Since the coefficients in (\ref{Ma23aa}) belongs to $C^{0,\alpha}$, we obtain that
all eigenfunctions belong to $C^{2,\alpha}(\overline{\Omega})$,
see \cite{LU}, Chapter 4.

\begin{lemma}\label{AL2} Let $\psi_x\neq 0$ in $\Omega_+$. Then the eigenvalues
$\mu_{1D}$ and $\mu_{1N}$ are simple and
$$
0>\mu_{1D}>\mu_{1N}.
$$
Furthermore,
$$
 \mu_{1N}=\mu_1,\;\;0=\mu_{2D}>\mu_{2N}
$$
  and
the corresponding
eigenfunction to $\mu_{2D}$ is odd and coincides with $\psi_x$. Moreover, $\mu_{3N}=\mu_2$
and $\mu_{3D}>0$.
\end{lemma}
\begin{proof} The proof of simplicity of $\mu_{1N}$ and $\mu_{1D}$ together with the
inequality
$\mu_{1D}>\mu_{1N}$ is quite standard and we omit it.

Let us show that $0>\mu_{1D}$. Consider the function $v=\psi_x$.
It satisfies the problem (\ref{Se28be}) with $\mu=0$. Since the
corresponding function is odd
it changes sign in $\Omega$. This implies $0>\mu_{1D}$. Moreover using that $v\neq 0$
in $\Omega_+$, we obtain that
$v$ is the second eigenfunction and $\mu_{2D}=0$ and the corresponding eigenfunction
is $v$.
Since the eigenfunctions corresponding to the eigenvalues
$\mu_1$ and $\mu_{1N}$ do not change sign in $\Omega$ we conclude $\mu_1=\mu_{1N}$.

The important property for the eigenfunctions of
the Dirichlet and Neumann problems for (\ref{Ma23aa}) is the following.
 One can verify that, if $w(x,y)$ is an eigenfunction
then $w(-x,y)$ is also an eigenfunction
corresponding to the same eigenvalue and hence the functions
$$
v_\pm=w(x,y)\pm w(-x,y)
$$
are also eigenfunctions or one of them can be zero. 
To proceed further it is convenient to introduce some more
eigenvalue problems. Considering the equations (\ref{Ma23aa}) only on the domain
$\Omega_+$
with Dirichlet or Neumann boundary conditions for $x=0$ and $x=\Lambda/2$ we get
four spectral problems with
DD, DN, ND and NN boundary conditions at $x=0$ and $x=\Lambda/2$. We denote corresponding
eigenvalues by
$$
\mu_{1DD}<\mu_{2DD}\leq\cdots,\;\;\mu_{1DN}<\mu_{2DN}\leq\cdots,\;\;
\mu_{1ND}<\mu_{2ND}\leq\cdots,\;\;
\mu_{1NN}<\mu_{2NN}\leq\cdots
$$
Clearly,
$$
\mu_{jDD}>\mu_{jDN}>\mu_{jNN}\;\;\mbox{and}\;\;\mu_{jDD}>\mu_{jND}>\mu_{jNN}\;\;
\mbox{for $j=1,2,\ldots$}
$$
Eigenfunctions corresponding to $\mu_{jDD}$ and $\mu_{jDN}$ can be extended for
all $x$ as odd,
and eigenfunctions corresponding to $\mu_{jND}$ and $\mu_{jNN}$ can be extended as
even.
Furthermore, all eigenvalues with the first index $1$ are simple and corresponding
eigenfunctions do not change sign in
$\Omega_+$ and
$$
\mu_{1D}=\mu_{1ND},\;\;\mu_{1N}=\mu_{1NN}=\mu_1<0,\;\;\mu_{2D}=\mu_{1DD}=0,\;\;
\mu_{2NN}=\mu_2>0
$$
The eigenvalue $\mu_{2N}$ coincides with $\mu_{2NN}$ or $\mu_{2DN}$, which is greater
 than $\mu_{2NN}$,  and hence
greater than $0$. The eigenvalue $\mu_{3D}$
coincides with $\mu_{2ND}$ or $\mu_{2NN}$ and hence also greater than $0$.
This completes the proof.

\end{proof}

\subsection{Proof of Theorem \ref{TOkt13a}}\label{SOkt22a}

\begin{proof}
We introduce an auxiliary spectral problem
\begin{eqnarray}\label{Ma24aa}
&&((\partial_x+z)^2+\partial_y^2)
u+\omega_* u+\mu u=0\;\;\mbox{in $D$},\nonumber\\
&&\nabla\hat{\psi}\cdot\nabla_z u-\sigma(x)u=0\;\;\mbox{for $y=\hat{\xi}(x)$},\nonumber\\
&&u(x,-d)=0,
\end{eqnarray}
where $z=i\tau$, $\tau$ is real and $w$ is $\Lambda$--periodic.
One can check that this is a spectral problem for self-adjoint operator for every
real $\tau$. We denote by
$$
\widehat{\mu}_1(\tau)\leq \widehat{\mu}_2(\tau)\leq\cdots
$$
its eigenvalues numerated in non-decreasing order and taking into account 
the multiplicities of eigenvalues.

Let
\begin{eqnarray*}
&&H^1_0(\Omega,\tau)=\{w=e^{i\tau x}v\,:\,v\in H^1_{0,\textrm{per}}(\Omega)\},\\
&&H^1_{0,\textrm{per}}(\Omega)=\{v\in H^1_0(\Omega)\,:\,v(-\Lambda/2,y)=v(\Lambda/2,y),
 y\in (-d,\xi(x))\}.
\end{eqnarray*}
These spaces coincide for $\tau=k\tau_*$, $k\in\Bbb Z$.
The above spectral problem can be reformulated as follows
\begin{equation}\label{Ma24b}
a(w,u)=\mu(w,u)_{L^2(\Omega)}\;\;\mbox{for all $v\in H^1_0(\Omega,\tau)$},
\end{equation}
where $w\in H^1_0(\Omega,\tau)$.

Let us prove that
\begin{equation}\label{J19a}
\mu_{jN}<\widehat{\mu}(\tau)_j<\mu_{jD},\;\;j=1,2,\ldots,
\end{equation}
for $\tau\in (0,\tau_*)$. Indeed, since $\tau\in (0,\tau_*)$,
$$
H^1_{0,\textrm{per}}(\Omega)\supset H^1_{0}(\Omega,\tau)
\supset H^1_{0,0}(\Omega)
$$
and both inclusions are strick.
Using the mini-max definition of the eigenvalues one can show
that
\begin{equation}\label{J19aa}
\mu_{jN}\leq\widehat{\mu}(\tau)_j\leq\mu_{jD},\;\;j=1,2,\ldots.
\end{equation}
To prove that the inequalities are strong assume that
$\mu_{jN}=\widehat{\mu}(\tau)_j$ for a certain $j$.
Then one can show that
there is an eigenfunction $v$, which simultaneously satisfies the Neumann and Dirichlet 
spectral problems
with the same eigenvalue. This implies that the eigenfunction
together with its normal derivative
vanish for $q=\pm\Lambda/2$. Hence the eigenfunction is identically zero.

The inequality (\ref{J19aa}) together with Lemma \ref{AL2} imply
$$
\widehat{\mu}_2(\tau)<0\;\;\mbox{and}\;\;\widehat{\mu}_3(\tau)>0,
$$
which in turn leads to the proof of our theorem.

\end{proof}

\subsection{Proof of Corollary \ref{CorOkt22}}\label{CorOkt22a}

Let $M$ be a positive integer and $v$ be a $(2M+1)\Lambda$--periodic function in $D$.
Introduce the following transformation
\begin{equation}\label{Okt22a}
V(v)(\tau_m,x,y)=V(\tau_m,x,y)=
\sum_{k=-M}^Me^{-i\tau_m(x+k\Lambda)}v(x+k\Lambda,y)
\end{equation}
and
\begin{equation}\label{Okt23a}
({\mathcal M}v)(x,y)=(V(\tau_{-M},x,y),V(\tau_{-M+1},x,y),\ldots,V(\tau_M,x,y)),
\end{equation}
where
$$
\tau_m=m\tau_1,\;\;\tau_1=\frac{2\pi}{(2M+1)\Lambda}.
$$
One can verify that the function $V(\tau_m,x,y)$ is periodic  with respect of $x$
 with the period $\Lambda$ for every  $m$ and
 ${\mathcal M}v$ is a vector function 
consisting of $2M+1$,  $\Lambda$--periodic functions.
 Moreover, if $v$ is an even function, then
 \begin{equation}\label{Okt24aa}
 V(\tau_m,-x,y)=V(-\tau_m,x,y)\;\;\mbox{for $m=0,\pm\tau_1,\ldots,\pm\tau_M$},
 \end{equation}
 and hence $V(\tau_0,x,y)$ is even with respect to $x$.

  We introduce one more operator by
 $$
 ({\mathcal N}V)(x,y)=\sum_{m=-M}^Me^{i\tau_mx}V(\tau_m,x,y),
 $$
 where $V$ is a $\Lambda$--periodic function with respect to $x$ for each $\tau_m$.
Applying the operator ${\mathcal N}$ to  $V$ given by (\ref{Okt22a}), we have
\begin{eqnarray*}
&&({\mathcal N}V)(x,y)=
\sum_{m=-M}^Me^{i\tau_mx}\sum_{k=-M}^Me^{-i\tau_m(x+k\Lambda)}v(x+k\Lambda,y)\\
&&=\sum_{k=-M}^Mv(x+k\Lambda,y)\sum_{m=-M}^M\Big(e^{-ik2\pi/(2M+1)}\Big)^m=
(2M+1)v(x,y).
\end{eqnarray*}
Thus the operator $(2M+1)^{-1}{\mathcal N}$ is inverse to ${\mathcal M}$.

To estimate norms of just introduced operators we note that
$$
\sum_{m=-M}^M||({\mathcal M}v)(\tau_m,x,y)||^2_{L^2(\Omega)}=(2M+1)||v||^2_{L^2(\Omega_M)},
$$
where
$$
\Omega_M=\{(x,y)\,:\,(-M-1/2)\Lambda<x<(M+1/2)\Lambda, \;-d<y<\xi(x)\}.
$$
Similar to \cite{N1} one can show
$$
\sum_{m=-M}^M||(V)(\tau_m,x,y)||^2_{H^2(\Omega)}\leq C_1||v||^2_{H^2(\Omega_M)}
$$
and
$$
||{\mathcal N}(V)||^2_{H^2(\Omega_M)}\leq
C_2\sum_{m=-M}^M||V(\tau_m,x,y)||^2_{H^2(\Omega)}.
$$

The important property of the transformation (\ref{Okt22a}) is the following
\begin{equation*}
A(x,y,\partial_x,\partial_y)({\mathcal N}V)(x,y)
=\sum_{m=-M}^Me^{i\tau_mx}A(x,y,\partial_x+i\tau_m,\partial_y)V(\tau_m,x,y)
\end{equation*}
and
\begin{equation*}
B(x,y,\partial_x,\partial_y)({\mathcal N}V)(x,y)
=\sum_{m=-M}^Me^{i\tau_mx}B(x,y,\partial_x+i\tau_m,\partial_y)V(\tau_m,x,y).
\end{equation*}
Therefore, if
\begin{eqnarray*}
&&A(x,y,\partial_x,\partial_y)({\mathcal N}V)(x,y)=0,\;\;\mbox{in $D$}\\
&&B(x,y,\partial_x,\partial_y)({\mathcal N}V)(x,y)=0
\end{eqnarray*}
then due to Theorem \ref{TOkt13a} $V(\tau_m,x,y)=0$ for all $m\neq 0$.
If $m=0$ then using (\ref{Okt24aa}) we get that the function $V(0,x,y)$ is even
and again the reference to Theorem \ref{TOkt13a} gives $V(0,x,y)=0$. This implies
${\mathcal N}V=0$. Thus the operator
$$
(A,B)\,:\, C^{2,\alpha}_{0,e,Mp}(D)\rightarrow C^{0,\alpha}_{0,e,Mp}(D)
\times C^{0,\alpha}_{e,Mp}(\Bbb R)
$$
is invertible.
Here the index $Mp$ indicate that we are dealing here with
$(2M+1)\Lambda$--periodic functions. Now small perturbation arguments lead
to invertibility of the corresponding nonlinear problem, which leads to the required result.

\subsection{Generalized eigenfunctions for the eigenvalue $z=0$}\label{SOkt25b}

The eigenvalue $z=0$ of the problem (\ref{Se28bd}) is simple and has the
eigenfunction $w=\psi_x$. To find generalized eigenfunctions is equivalent to finding
of the Floque solutions to the problem (\ref{Okt26a}). Then the generalized
eigenfunctions are coefficients in the Floquet solutions.

The first generalized eigenfunction $u_1$ is the coefficient in the Floquet solution
$$
u=xu_0+ u_1,\;\;u_0=\psi_x.
$$
Then
\begin{eqnarray}\label{Okt17a}
&&(\Delta+\omega_*)u_1+2u_{0x}=0\;\;\mbox{in $D$},\nonumber\\
&&\nabla\psi\cdot\nabla u_1-\sigma u_1+\psi_xu_{0}=0\;\;\mbox{for $y=\xi(x)$}.
\end{eqnarray}
The function $u_1$ is even, $\Lambda$-periodic solution to (\ref{Okt17a}) satisfying
$u_1(x,-d)=0$. Since $\mu_1<0$ and $\mu_2>0$ this problem is uniquely solvable.

Let us find $u_1$. We assume that we have a branch of solutions parameterized by the
period $\Lambda$, which depends on $t$ now. We introduce
$$
\lambda=\frac{\Lambda(0)}{\Lambda(t)}
$$
and make the change of variables
$$
X=\lambda x, \;\;\partial_x=\lambda\partial_X.
$$
Then the problem can be rewritten as a problem with the same period $\Lambda_0$ as
\begin{eqnarray}\label{Okt18a}
&&(\lambda^2\partial_X^2+\partial_y^2+\omega_*)u_1+2\lambda u_{0X}=0\;\;
\mbox{in $\widehat{D}$},\nonumber\\
&&(\lambda^2\psi_X u_{1X}+\psi_yu_{1y}-\sigma u_1+2\lambda\psi_Xu_{0}=0\;\;
\mbox{for $y=\xi(X)$}
\end{eqnarray}
Consider the problem for $\psi$ in $(X,y)$ variables:
\begin{eqnarray*}
&&(\lambda^2\partial_X^2+\partial_y^2)\psi+\omega(\psi)=0\;\;\mbox{in $\widehat{D}$}\\
&&\frac{1}{2}(\lambda^2\psi_x^2+\psi_y^2)+\xi=R\;\;\mbox{for $y=\xi(X,t)$}\\
&&\psi(X,\xi(X),t)=0,\;\;\psi(X,-d)=p_0.
\end{eqnarray*}
Differentiating these relations with respect to $t$, we get
\begin{eqnarray*}
&&(\lambda^2\partial_X^2+\partial_y^2)\psi_t+\omega'(\psi)\psi_t=-2\lambda\lambda'
\psi_{XX}\;\;
\mbox{in $\widehat{D}$}\\
&&(\lambda^2\psi_x\psi_{tx}+\psi_y\psi_{ty})+(\lambda^2\psi_x\psi_{xy}+\psi_y\psi_{yy})
\xi_t+\xi_t=-\lambda\lambda'\psi_{X}^2\;\;\mbox{for $y=\xi(X,t)$}\\
&&\psi_t(X,\xi(X),t)+\psi_y\xi_t=0,\;\;\psi_t(X,-d)=0.
\end{eqnarray*}
It can be transformed to
\begin{eqnarray*}
&&(\lambda^2\partial_X^2+\partial_y^2)\psi_t+\omega'(\psi)\psi_t=-2\lambda\lambda'
\psi_{XX}\;\;
\mbox{in $\widehat{D}$}\\
&&(\lambda^2\psi_x\psi_{tx}+\psi_y\psi_{ty})-\sigma_\lambda \psi_t
=-\lambda\lambda'\psi_x^2\;\;\mbox{for $y=\xi(X,t)$}\\
&&\psi_t(X,-d)=0,
\end{eqnarray*}
where
$$
\sigma_\lambda=-\frac{\lambda^2\psi_x\psi_{xy}+\psi_y\psi_{yy}+1}{\psi_y}.
$$
Therefore
$$
u_1=\frac{1}{\lambda\lambda'}\psi_t.
$$

The second generalized eigenvector $u_2$ is the last term in possible solution
$$
\frac{x^2}{2}u_0+xu_1+u_2.
$$
Then $u_2$ must satisfy
\begin{eqnarray}\label{Okt17aa}
&&(\Delta+\omega_*)u_2+u_{0}+2u_{1x}=0\;\;\mbox{in $D$},\nonumber\\
&&\nabla\psi\cdot\nabla u_2-\sigma u_2+\psi_xu_{1}=0\;\;\mbox{for $y=\xi(x)$},\\
&&u_2(x,-d)=0.
\end{eqnarray}
The function must be $u_2$ is odd and $\Lambda$-periodic.
The solvability condition for (\ref{Okt17aa}) is the following
\begin{equation}\label{Okt19a}
\int_\Omega (u_0+2u_{1x})u_0dxdy-\int_S\psi_xu_{1}u_0\frac{dx}{\psi_y}=0.
\end{equation}
Consider an analytic branch of water waves starting from a laminar flow,
where the period is considered as a parameter. Near a laminar flow $(\hat{psi}(y),0)$ the
branch is described by
$$
\psi(x,y,t)=\hat{psi}(y)+t\cos(\tau_*x)\gamma(y,\tau_*)+O(t^2),
$$
$$
\xi(x,t)=t\frac{\tau_*\sin(\tau_*x)}{\hat{\psi}_y(y)}\gamma(0,\tau_*)+O(t^2)
$$
and
$$
\lambda=1-ct^2+O(t^4).
$$
Therefore
$$
u_0=-t\tau_*\sin(\tau_*x)\gamma(y,\tau_*)+O(t^2),
$$
$$
u_1=-\frac{1}{2ct}\cos(\tau_*x)\gamma(y,\tau_*)+O(1)
$$
and
$$
\lambda'=-2ct+O(t^3).
$$
Since the left-hand side (LHS) of (\ref{Okt19a}) is
$$
LHS=-\frac{\tau_*}{2c}\int_\omega \sin^2(\tau_*x)\gamma^2)y,\tau_*)dxdy+O(t),
$$
it is different from zero for small $t$. Since the function $LHS(t)$ is analytic, it does
not vanish except some isolated $t$-points.

\section*{Acknowledgments}

This work was supported by the Ministry of Science and Higher Education of
the Russian Federation (agreement 075-15-2025-344 dated 29/04/2025 for
Saint Petersburg Leonard Euler International Mathematical Institute at PDMI RAS).

\section{References}

{

\end{document}